\documentclass[11pt]{amsart}
\usepackage{amscd}
\usepackage{amsfonts}
\usepackage{amsmath}
\usepackage{amssymb}
\usepackage{bbm}
\usepackage{blkarray}
\usepackage{enumerate}
\usepackage{extarrows}
\usepackage{fancybox}
\usepackage{graphicx}
\usepackage{indentfirst}
\usepackage{mathbbol}
\usepackage{mathrsfs}
\usepackage{mathtools}
\usepackage{picinpar}
\usepackage{url}
\usepackage{varioref}
\usepackage{wrapfig}
\usepackage{xcolor}
\usepackage[colorlinks, linkcolor=black, anchorcolor=black, citecolor=black, urlcolor=black]{hyperref}

\newtheorem{thm}{Theorem}[section]
\newtheorem{prop}[thm]{Proposition}
\newtheorem{lem}[thm]{Lemma}

\theoremstyle{definition}

\theoremstyle{remark}
\newtheorem*{rmk}{Remark}
\theoremstyle{definition}
\newtheorem{defn}[thm]{Definition}

\begin{document}

\title[]{A generalization of Nadel vanishing theorem}

\date{}

\thanks{}

\author[Xiankui Meng]{Xiankui Meng}
\address{Xiankui Meng: School of Science, Beijing University of Posts and Telecommunications, Beijing 100876, China.}
\email{mengxiankui@amss.ac.cn}

\author[Xiangyu Zhou]{Xiangyu Zhou}
\address{Xiangyu Zhou: Institute of Mathematics, AMSS, and Hua Loo-Keng Key Laboratory of Mathematics, Chinese Academy of Sciences, Beijing 100190, China.}
\email{xyzhou@math.ac.cn}

\subjclass[2010]{}


\begin{abstract}
In this paper we first prove a version of $L^{2}$ existence theorem for line bundles equipped a singular Hermitian metrics. 
Aa an application, we establish a vanishing theorem which generalizes the classical Nadel vanishing theorem.
\end{abstract}

\maketitle


\section{Introduction}

The fundamental $L^{2}$ estimates concerning solutions of the Cauchy-Riemann operators are essentially due to H\"ormander (\cite{Hor65}).
The best way of expressing these $L^{2}$ estimates is to use a geometric setting first considered by Andreotti-Vesentini (\cite{An-Ve65}).
In \cite{Oh84}, Ohsawa proved his $L^{2}$ vanishing theorem for smooth Hermitian vector bundle $(E,h)$ in order to obtain his cohomology vanishing theorems. 
By regularization arguments, we will show that Ohsawa's $L^{2}$ vanishing theorem is also valid when $E$ is a line bundle and $h$ is a singular Hermitian metric with positive curvature in the sense of currents. 

A singular Hermitian metric $h$ on a holomorphic line bundle $L\longrightarrow X$ is a metric which is given in any local trivialization
$\theta: L\upharpoonright\Omega\longrightarrow\Omega\times\mathbb{C}$ by
\begin{equation*}
  \|\xi\|_{h}=|\theta(\xi)|e^{-\varphi(x)}, \quad x\in\Omega, \quad \xi\in L_{x},
\end{equation*}
where $\varphi\in L_{loc}^{1}(\Omega)$ is an arbitrary function,
called the weight of the metric with respect to the trivialization $\theta$.
A complex manifold $X$ is said to be weakly pseudoconvex if there exists a smooth psh function $\psi$ on $X$ such that for any $c\in\mathbb{R}$, all level sets
\[X_{c}=\{x\in X;\quad \psi(x)<c\}\]
are relatively compact.
The first result of this paper is the following theorem.
\begin{thm}\label{thm:intro1}
Let $X$ be a weakly pseudoconvex K\"ahler manifold of dimension $n$ 
and let $\sigma$ be a $d$-closed semi-positive $(1,1)$ form on $X$.
Let $L$ be a line bundle over $X$ and let $h$ be a singular hermitian metic on $L$ such that its curvature current $\mathrm{i}\Theta_{L,h}\geqslant\alpha\sigma$ (in the sense of distributions) for some continuous positive function $\alpha$ on $X$. Suppose $q\geqslant 1$, then for any $\bar{\partial}$-closed $L$-valued $(n,q)$ form $f$ which is local square integrable and
\[\lim_{\varepsilon\to 0}\int_{X}\frac{1}{\alpha}\left|f\right|^{2}_{\sigma+\varepsilon\omega,h}\mathrm{d}V_{\sigma+\varepsilon\omega}<+\infty.
\]
we can find an $L$-valued $(n,q-1)$ form $u$ such that $\bar{\partial}u=f$ and
\begin{equation}
\left\|u\right\|_{\sigma,h}^{2}=\lim_{\varepsilon\to0}\int_{X}\left|u\right|^{2}_{\sigma+\varepsilon\omega,h}\mathrm{d}V_{\sigma+\varepsilon\omega}\leqslant\lim_{\varepsilon\to 0}\int_{X}\frac{1}{q\alpha}\left|f\right|^{2}_{\sigma+\varepsilon\omega,h}\mathrm{d}V_{\sigma+\varepsilon\omega}.
\end{equation}
\end{thm}

The second part of this paper is to consider sheaf cohomology groups with multiplier ideal sheaves. Vanishing theorems are important in both complex geometry and algebraic geometry. Various generalizations of the classical Kodaira vanishing theorem and the Nadel vanishing theorem are great developments in this direction.
For general Nadel type vanishing theorem, one may refer to \cite{Cao14,Guan-Zhou15a}.

Let $\psi$ be a quasi-psh function on a complex manifold $X$.
The multiplier ideal sheaf $\mathcal{I}(\psi)$
is the ideal subsheaf of $\mathcal{O}_{X}$ defined by
$$\mathcal{I}(\psi)_{x}=\{f\in\mathcal{O}_{X,x};~\exists ~U \ni x ~\text{such that}~\int_{U} |f|^{2}e^{-2\psi}d\lambda<\infty \},$$
where $U$ is an open coordinate neighborhood of $x$,
and $d\lambda$ is the standard Lebesgue measure in $\mathbb{C}^{n}$.
It's well-known that $\mathcal{I}(\psi)$ is a coherent analytic sheaf.
Let $h$ be a singular Hermitian metric on $L\longrightarrow X$ with curvature current $\mathrm{i}\Theta\geqslant 0$. Then we can write $h=h_{0}e^{-2\psi}$, where $h_{0}$ is a smooth Hermitian metric on $L$ and $\psi$ is quasi-psh. We can define $\mathcal{I}(h)=\mathcal{I}(\psi)$.
Using Theorem \ref{thm:intro1}, we can prove the following statement which generalizes Ohsawa's Vanishing theorem (Theorem 3.1 in \cite{Oh84}) and Nadel vanishing theorem.

\begin{thm}\label{thm:vanishing}
Let $X$ be a weakly pseudoconvex K\"ahler manifold.
let $f:X\to Y$ be a proper holomorphic map to a paracompact complex space $Y$ with a K\"ahler metric $\sigma$, and let $L$ be a line bundle over X equipped with a singular Hermitian metric. Assume that $\mathrm{i}\Theta_{L,h}\geqslant\varepsilon\cdot f^{\ast}\sigma$ (in the sense of current) for some continuous positive function $\varepsilon$ on $X$, then
\[H^{q}\left(Y,f_{\ast}\left(K_{X}\otimes L\otimes\mathcal{I}(h)\right)\right)=0\quad\text{for}\quad q\geqslant 1.\]
\end{thm}
For a related vanishing theorem, see \cite{Kol86,Mat16,Fuj18}.


\section{Preliminaries}

Let $(X,\omega)$ be a complex $n$-dimensional complete Hermitian manifold and let $L$ be a holomorphic line bundle over $X$ equipped with a (singular) Hermitian metric $h$. 
We denote by $L^{2}_{n,q}(X,L,\omega,h)$ the space of measurable $L$-valued forms of type $(n,q)$ which are square integrable with respect to $h$ and $\omega$. 
For $u,v\in L^{2}_{n,q}(X,L,\omega,h)$, the inner product 
\[\left(u,v\right)_{\omega,h}=\int_{X}\langle u,v\rangle_{\omega,h}dV_{\omega},\] where $\langle u,v\rangle_{\omega,h}$ denotes the pointwise inner product of $u$ and $v$ with respect to $\omega$ and $h$, and $dV_{\omega}$ the volume form with respect to $\omega$.
By $\mathcal{C}^{n,q}_{0}(X,L)$, we denote the space of $L$-valued smooth $(n,q)$-forms with compact support. 

Let $D$ be the Chern connection of $(L,h)$ which can be split in a unique way as $D=D'+\bar{\partial}$ and let $\Theta_{L,h}$ be its curvature tensor (current).
When $\omega$ is complete and $h$ is smooth, one can view $D, D', \bar{\partial}$ as closed and densely defined operators between Hilbert spaces $L^{2}_{n,q}(X,L,\omega,h)$ and the formal adjoint $\bar{\partial}^{\ast}$ coincide with the Hilbert adjoint.
The following facts are basic.
\begin{prop}
Let $\omega_{1}$ and $\omega_{2}$ be two hermitian metrics satisfying $\omega_{1}\geqslant\omega_{2}$. Then for any $u\in \mathcal{C}^{n,q}_{0}(X,L)$, we have
\[\left\|u\right\|_{\omega_{1},h}\leqslant\left\|u\right\|_{\omega_{2},h}.\]
\end{prop}

\begin{defn}
Given a smooth semipositive $(1,1)$-form $\sigma$ on $X$, we set
\[L^{2}_{n,q}(X,L,\sigma,h)=\left\{u\in L^{2}_{n,q}(X,L,\omega+\sigma,h);\quad \lim_{\varepsilon\to 0}\left\|u\right\|_{\varepsilon\omega+\sigma,h}~\text{exists}\right\},\]
and
\[\left\|u\right\|_{\sigma,h}=\lim_{\varepsilon\to 0}\left\|u\right\|_{\varepsilon\omega+\sigma,h}.\]
$L^{2}_{n,q}(X,L,\sigma,h)$ is a Hilbert space with norm $\left\|u\right\|_{\sigma,h}$. 
\end{defn}

This space was introduced by Ohsawa in \cite{Oh84} and it can be seen as the $L^{2}$ space with respect to the pseudo-metric $\sigma$. 

\begin{prop}[\cite{Oh84}]
$L^{2}_{n,q}(X,L,\sigma,h)$ and $\left\|u\right\|_{\sigma,h}$ do not depend on the
choice of the metric $\omega$.
\end{prop}

\begin{prop}[Bochner-Kodaira-Nakano identity]\label{prop:BKNE}
Let $(X,\omega)$ be a complete K\"{a}hler manifold and $(L,h)$ a smooth Hermitian holomorphic line bundle such that the curvature possesses a uniform lower bound $i\Theta_{L,h}\geqslant-C\omega$. For every measurable $(n,q)$-form $u$ with $L^{2}$ coefficients and values in $L$ such that its differentials $\bar{\partial}u$, $\bar{\partial}^{\ast}u$ also in $L^{2}$, then
\begin{equation}\label{prop:BKN}
\|\bar{\partial}u\|_{\omega,h}^{2}+\|\bar{\partial}^{\ast}u\|_{\omega,h}^{2}=\|D'u\|_{\omega,h}^{2}+\left(\mathrm{i}\Theta_{L,h}\Lambda_{\omega} u,u\right)_{\omega,h}
\end{equation}
\end{prop}

\begin{lem}[cf. \cite{Dembook}]\label{lem:extension}
Let $\Omega$ be an open subset of $\mathbb{C}^{n}$ and $Y$ an analytic subset of $\Omega$. Assume that $v$ is a $(p, q-1)$-form with $L^{2}_{loc}$ coefficients and $w$ a $(p, q)$-form with $L^{1}_{loc}$ coefficients such that $\bar{\partial}v = w$ on $\Omega\backslash Y$ (in the sense of distribution theory). Then $\bar{\partial}v = w$ on $\Omega$.
\end{lem}

\begin{thm}[\cite{Dem94}]\label{thm:approximation}
Let $(X, \omega)$ be a complex manifold equipped with a Hermitian metric $\omega$, and $\Omega\Subset X$ be an open subset. 
Assume that $T=\widetilde{T}+\frac{\mathrm{i}}{\pi}\partial\bar{\partial}\varphi$
is a closed $(1,1)$-current on $X$, 
where $\widetilde{T}$ is a smooth real $(1, 1)$-form and $\varphi$ is a quasi-plurisubharmonic function. 
Let $\gamma$ be a continuous real $(1,1)$-form such that $T\geqslant\gamma$. Suppose that the Chern curvature tensor of $TX$ satisfies
\begin{equation*}
\left(\mathrm{i}\Theta_{T_{X}}+\theta\otimes\mathrm{Id}_{T_{X}}\right)\left(\kappa_{1}\otimes\kappa_{2},\kappa_{1}\otimes\kappa_{2}\right)\geqslant 0
\quad \forall\kappa_{1},\kappa_{2}\in T_{X}~\text{with}~\langle\kappa_{1},\kappa_{2}\rangle=0
\end{equation*}
for some continuous nonnegative $(1,1)$-form $\theta$ on $X$. 
Then there is a family of closed $(1,1)$-currents $T_{\varsigma,\varepsilon}=\widetilde{T}+\frac{\mathrm{i}}{\pi}\partial\bar{\partial}\varphi_{\varsigma,\varepsilon}$ on $X$ such that
\begin{enumerate}[(i)]
\item
$\varphi_{\varsigma,\varepsilon}$ is quasi-plurisubharmonic on a neighborhood of $\bar{\Omega}$, smooth on $X\backslash E_{\sigma}(T)$, increasing with respect to $\varsigma$ and $\varepsilon$ on $\Omega$, and converges to $\varphi$ on $\Omega$ as $\varepsilon\to 0$,
\item
$T_{\varsigma,\varepsilon}\geqslant\gamma-\varsigma\theta-\delta_{\varepsilon}\omega$ on $\Omega$,
\end{enumerate}
where $E_{\varsigma}(T)=\left\{x\in X; \nu(T,x)\geqslant\varsigma\right\}$, $\varsigma>0$, is the $\varsigma$-upperlevel set of Lelong numbers, and $\{\delta_{\varepsilon}\}$ is an increasing family of positive numbers such that $\lim_{\varepsilon\to 0}\delta_{\varepsilon}=0$.
\end{thm}


\section{$L^{2}$-existence theorem}

\begin{thm}\label{thm:existence}
Let $X$ be a weakly pseudoconvex K\"ahler manifold of dimension $n$ 
and let $\sigma$ be a $d$-closed semi-positive $(1,1)$ form on $X$.
Let $L$ be a line bundle over $X$ and let $h$ be a singular hermitian metic on $L$ such that its curvature current $\mathrm{i}\Theta_{L,h}\geqslant\alpha\sigma$ (in the sense of distributions) for some continuous positive function $\alpha$ on $X$. Suppose $q\geqslant 1$, then for any $\bar{\partial}$-closed $L$-valued $(n,q)$ form $f$ which is local square integrable and
\[\lim_{\varepsilon\to 0}\int_{X}\frac{1}{\alpha}\left|f\right|^{2}_{\sigma+\varepsilon\omega,h}\mathrm{d}V_{\sigma+\varepsilon\omega}<+\infty.
\]
we can find an $L$-valued $(n,q-1)$ form $u$ such that $\bar{\partial}u=f$ and
\begin{equation}
\left\|u\right\|_{\sigma,h}^{2}=\lim_{\varepsilon\to0}\int_{X}\left|u\right|^{2}_{\sigma+\varepsilon\omega,h}\mathrm{d}V_{\sigma+\varepsilon\omega}\leqslant\lim_{\varepsilon\to 0}\int_{X}\frac{1}{q\alpha}\left|f\right|^{2}_{\sigma+\varepsilon\omega,h}\mathrm{d}V_{\sigma+\varepsilon\omega}.
\end{equation}
\end{thm}

\begin{proof}
Let $\omega$ be a complete K\"ahler metric on $X$ (every weakly pseudoconvex K\"{a}hler manifold admits a complete K\"{a}ler metric) and let $\omega_{\varepsilon}=\sigma+\varepsilon\omega$.
Let $h_{0}$ be a smooth hermitian metric on $L$, then $h=h_{0}e^{-\varphi}$,
where $\varphi$ is quasi-psh function on $X$ and
\[\mathrm{i}\Theta_{L,h_{0}}+\mathrm{i}\partial\bar{\partial}\varphi\geqslant\alpha\sigma.\]
By Theorem \ref{thm:approximation}, there is a sequence of quasi-psh functions $\{\varphi_{\nu}\}$ defined on $X_{c}$ such that
\begin{enumerate}[(1)]
 \item $\varphi_{\nu}$ is smooth in the complement $X_{c}\backslash Z_{\nu}$ of an analytic set $Z_{\nu}\subset X_{c}$;
\item $\{\varphi_{\nu}\}$ is a decreasing sequence and $\varphi|_{X_{c}}=\lim_{\nu\rightarrow+\infty}\varphi_{\nu}$;
\item $\mathrm{i}\Theta_{L,h_{0}}+\mathrm{i}\partial\bar{\partial}\varphi_{\nu}\geqslant\alpha\sigma-\beta_{\nu}\omega$ on $X_{c}$, where $\lim_{\nu\rightarrow+\infty}\beta_{\nu}=0$.
\end{enumerate}
Now we can find a sequence of hermitian metric
\[h_{\nu}=h_{0}e^{-\varphi_{\nu}}.\]
 on $L|_{X_{c}}$. Then $h_{v}$ is smooth on $X_{c}\backslash Z_{\nu}$, $h_{v}\leqslant h$ and 
 \begin{equation}\label{eq:cur-nu}
 \mathrm{i}\Theta_{L,h_{\nu}}\geqslant\alpha\sigma-\beta_{\nu}\omega.
 \end{equation}
We can choose a complete K\"{a}hler metric 
\[\omega_{c}=2\omega+\mathrm{i}\partial\bar{\partial}\psi_{\nu}+\mathrm{i}\partial\bar{\partial}\log(c-\psi)^{-1}\]
on $X_{c}\backslash Z_{\nu}$, where $\psi_{\nu}$ is a quasi-psh function on $X_{c}$, $\psi_{\nu}$ smooth on $X_{c}\backslash Z_{\nu}$ and $\omega+\mathrm{i}\partial\bar{\partial}\psi_{\nu}\geqslant 0$ (see e.g. \cite{Dem82}). 
Then we define a family
\[\omega_{\varepsilon,\delta}=\sigma+\varepsilon\omega+\delta\omega_{c},\quad \delta>0\]
of complete K\"ahler metric on $X_{c}\backslash Z_{\nu}$.

Let $g\in\mathcal{C}^{n,q}_{0}(X_{c}\backslash Z_{\nu},L)$ 
and let $\varepsilon'<\varepsilon$, $\delta'<\delta$ be two positive numbers.
Let $x\in X_{c}\backslash Z_{\nu}$ be any point, the operator
\begin{equation}\label{eq:defA}
A=\alpha\omega_{\varepsilon',\delta'}\Lambda_{\omega_{\varepsilon,\delta}}
\end{equation}
is positive on $\wedge^{n,q}T^{\ast}_{X,x}\otimes L_{x}$.
We express $\omega_{\varepsilon,\delta}$, $\omega_{\varepsilon',\delta'}$ and $g$ at $x$ as follows:
\begin{equation}
\left\{
\begin{split}
&\omega_{\varepsilon,\delta}=\mathrm{i}\sum_{j=1}^{n}\tau_{j}\wedge\bar{\tau}_{j}, \\
&\omega_{\varepsilon',\delta'}=\mathrm{i}\sum_{j=1}^{n}\lambda_{j}\tau_{j}\wedge\bar{\tau}_{j},\quad 0<\lambda_{j}<1, \\
&g=\sum_{j_{1}<\cdots<j_{q}}g_{j_{1},\cdots,j_{q}}\tau_{1}\wedge\cdots\wedge\tau_{n}\wedge\bar{\tau}_{j_{1}}\wedge\cdots\bar{\tau}_{j_{q}}.
\end{split}\right.
\end{equation}
Then we have
\begin{equation}
\langle Ag,g\rangle_{\omega_{\varepsilon,\delta},h_{\nu}}=\sum_{j_{1}<\cdots<j_{q}}\alpha(\lambda_{j_{1}}+\cdots+\lambda_{j_{q}})\left|g_{j_{1},\cdots,j_{q}}\right|_{h_{\nu}}^{2},
\end{equation}
and hence
\begin{equation}\label{eq:A_inv}
\langle A^{-1}g,g\rangle_{\omega_{\varepsilon,\delta},h_{\nu}}\mathrm{d}V_{\omega_{\varepsilon,\delta}}=\sum_{j_{1}<\cdots<j_{q}}\frac{1}{\alpha\left(\lambda_{j_{1}}+\cdots+\lambda_{j_{q}}\right)}\left|g_{j_{1},\cdots,j_{q}}\right|_{h_{\nu}}^{2}\mathrm{d}V_{\omega_{\varepsilon,\delta}}.
\end{equation}
We can write
\begin{equation}\label{eq:g_snorm}
\langle g,g\rangle_{\omega_{\varepsilon',\delta'},h_{\nu}}\mathrm{d}V_{\omega_{\varepsilon',\delta'}}=\sum_{j_{1}<\cdots<j_{q}}\frac{1}{\prod_{k=1}^{q}\lambda_{j_{k}}} \left|g_{j_{1},\cdots,j_{q}}\right|_{h_{\nu}}^{2}\mathrm{d}V_{\omega_{\varepsilon,\delta}}.
\end{equation}
Comparing \eqref{eq:A_inv} and \eqref{eq:g_snorm}, we have
\begin{equation}
\langle A^{-1}g,g\rangle_{\omega_{\varepsilon,\delta},h_{\nu}}\mathrm{d}V_{\omega_{\varepsilon,\delta}}\leqslant\frac{1}{q\alpha}\langle g,g\rangle_{\omega_{\varepsilon',\delta'},h_{\nu}}\mathrm{d}V_{\omega_{\varepsilon',\delta'}},
\end{equation}
and the integrals
\begin{equation}
\int_{X_{c}\backslash Z_{\nu}}\langle A^{-1}g,g\rangle_{\omega_{\varepsilon,\delta},h_{\nu}}\mathrm{d}V_{\omega_{\varepsilon,\delta}}
\leqslant\frac{1}{q}\int_{X_{c}}\frac{1}{\alpha}\langle g,g\rangle_{\omega_{\varepsilon',\delta'},h_{\nu}}\mathrm{d}V_{\omega_{\varepsilon',\delta'}}.
\end{equation}
Therefore
 \begin{equation}
 \begin{split}
\int_{X_{c}\backslash Z_{\nu}}\langle A^{-1}f,f\rangle_{\omega_{\varepsilon,\delta},h_{\nu}}\mathrm{d}V_{\omega_{\varepsilon,\delta}}
&\leqslant\frac{1}{q}\int_{X_{c}}\frac{1}{\alpha}\langle f,f\rangle_{\omega_{\varepsilon',\delta'},h_{\nu}}\mathrm{d}V_{\omega_{\varepsilon',\delta'}} \\
&\leqslant\frac{1}{q}\int_{X_{c}}\frac{1}{\alpha}\langle f,f\rangle_{\sigma+\varepsilon'\omega,h_{\nu}}\mathrm{d}V_{\sigma+\varepsilon'\omega} \\
&\leqslant\frac{1}{q}\int_{X}\frac{1}{\alpha}\langle f,f\rangle_{\sigma+\varepsilon'\omega,h}\mathrm{d}V_{\sigma+\varepsilon'\omega} \\
&\leqslant\frac{C}{q},
\end{split}
\end{equation}
where the constant
\[C=\lim_{\varepsilon\to 0}\int_{X}\frac{1}{\alpha}\langle f,f\rangle_{\sigma+\varepsilon\omega,h}\mathrm{d}V_{\sigma+\varepsilon\omega}<+\infty\]
For every $v\in L_{n,q}^{2}\left(X_{c}\backslash Z_{\nu},L,\omega_{\varepsilon,\delta},h_{\nu}\right)$ we have
 \begin{equation}
 \begin{split}
 &\left|(f,v)_{\omega_{\varepsilon,\delta},h_{\nu}}\right|^{2} \\
 =&\left|\int_{X_{c}\backslash Z_{\nu}}\langle f,v\rangle_{\omega_{\varepsilon,\delta},h_{\nu}}\mathrm{d}V_{\omega_{\varepsilon,\delta}}\right|^{2} \\
\leqslant &\left|\int_{X_{c}\backslash Z_{\nu}}
 \langle A^{-1}f,f\rangle_{\omega_{\varepsilon,\delta},h_{\nu}}^{\frac{1}{2}}\cdot\langle Av,v\rangle_{\omega_{\varepsilon,\delta},h_{\nu}}^{\frac{1}{2}}\mathrm{d}V_{\omega_{\varepsilon,\delta}}\right|^{2} \\
\leqslant&\int_{X_{c}\backslash Z_{\nu}}\langle A^{-1}f,f\rangle_{\omega_{\varepsilon,\delta},h_{\nu}}\mathrm{d}V_{\omega_{\varepsilon,\delta}}\int_{X_{c}\backslash Z_{\nu}}\langle Av,v\rangle_{\omega_{\varepsilon,\delta},h_{\nu}}\mathrm{d}V_{\omega_{\varepsilon,\delta}} \\
\leqslant&\frac{C}{q}\int_{X_{c}\backslash Z_{\nu}}\langle Av,v\rangle_{\omega_{\varepsilon,\delta},h_{\nu}}\mathrm{d}V_{\omega_{\varepsilon,\delta}}.
\end{split}
\end{equation}
by means of Cauchy-Schwarz inequality.
The operator $A$ is defined by \eqref{eq:defA} and hence
 \begin{equation}
 \begin{split}
\int_{X_{c}\backslash Z_{\nu}}\langle Av,v\rangle_{\omega_{\varepsilon,\delta},h_{\nu}}\mathrm{d}V_{\omega_{\varepsilon,\delta}}
=&\int_{X_{c}\backslash Z_{\nu}}\langle \alpha\sigma\Lambda_{\omega_{\varepsilon,\delta}}v,v\rangle_{\omega_{\varepsilon,\delta},h_{\nu}}\mathrm{d}V_{\omega_{\varepsilon,\delta}} \\
&+\varepsilon'\int_{X_{c}\backslash Z_{\nu}}\langle\alpha\omega\Lambda_{\omega_{\varepsilon,\delta}}v,v\rangle_{\omega_{\varepsilon,\delta},h_{\nu}}\mathrm{d}V_{\omega_{\varepsilon,\delta}} \\
&+\delta'\int_{X_{c}\backslash Z_{\nu}}\langle\alpha\omega_{c}\Lambda_{\omega_{\varepsilon,\delta}}v,v\rangle_{\omega_{\varepsilon,\delta},h_{\nu}}\mathrm{d}V_{\omega_{\varepsilon,\delta}}.
\end{split}
\end{equation} 
The above two inequalities imply
 \begin{equation}
 \begin{split}
\left|(f,v)_{\omega_{\varepsilon,\delta},h_{\nu}}\right|^{2}
\leqslant & \frac{C}{q}\int_{X_{c}\backslash Z_{\nu}}\langle \alpha\sigma\Lambda_{\omega_{\varepsilon,\delta}}v,v\rangle_{\omega_{\varepsilon,\delta},h_{\nu}}\mathrm{d}V_{\omega_{\varepsilon,\delta}} \\
&+\varepsilon'\frac{C}{q}\int_{X_{c}\backslash Z_{\nu}}\langle\alpha\omega\Lambda_{\omega_{\varepsilon,\delta}}v,v\rangle_{\omega_{\varepsilon,\delta},h_{\nu}}\mathrm{d}V_{\omega_{\varepsilon,\delta}} \\
&+\delta'\frac{C}{q}\int_{X_{c}\backslash Z_{\nu}}\langle\alpha\omega_{c}\Lambda_{\omega_{\varepsilon,\delta}}v,v\rangle_{\omega_{\varepsilon,\delta},h_{\nu}}\mathrm{d}V_{\omega_{\varepsilon,\delta}}.
\end{split}
\end{equation} 
Letting $\varepsilon'\to 0$ and $\delta'\to 0$, we have
 \begin{equation}
 \left|(f,v)_{\omega_{\varepsilon,\delta},h_{\nu}}\right|^{2}\leqslant\frac{C}{q}\int_{X_{c}\backslash Z_{\nu}}\langle \alpha\sigma\Lambda_{\omega_{\varepsilon,\delta}}v,v\rangle_{\omega_{\varepsilon,\delta},h_{\nu}}\mathrm{d}V_{\omega_{\varepsilon,\delta}}
 \end{equation} 
By the construction of $h_{\nu}$, 
 \begin{equation}
 \alpha\sigma\leqslant\mathrm{i}\Theta_{L,h_{\nu}}+\beta_{\nu}\omega\leqslant\mathrm{i}\Theta_{L,h_{\nu}}+\beta_{\nu}\omega_{\varepsilon,\delta}.
 \end{equation}
Therefore,
 \begin{equation}
 \begin{split}
\int_{X_{c}\backslash Z_{\nu}}\langle \alpha\sigma\Lambda_{\omega_{\varepsilon,\delta}}v,v\rangle_{\omega_{\varepsilon,\delta},h_{\nu}}\mathrm{d}V_{\omega_{\varepsilon,\delta}}\leqslant
&\int_{X_{c}\backslash Z_{\nu}}\langle \mathrm{i}\Theta_{L,h_{\nu}}\Lambda_{\omega_{\varepsilon,\delta}}v,v\rangle_{\omega_{\varepsilon,\delta},h_{\nu}}\mathrm{d}V_{\omega_{\varepsilon,\delta}} \\
&+q\beta_{\nu}\int_{X_{c}\backslash Z_{\nu}}\langle v,v\rangle_{\omega_{\varepsilon,\delta},h_{\nu}}\mathrm{d}V_{\omega_{\varepsilon,\delta}}.
\end{split}
 \end{equation} 
We thus obtain
\begin{equation}
 \left|(f,v)_{\omega_{\varepsilon,\delta},h_{\nu}}\right|^{2}\leqslant\beta_{\nu}C\|v\|_{\omega_{\varepsilon,\delta},h_{\nu}}^{2}+\frac{C}{q}\left(\mathrm{i}\Theta_{L,h_{\nu}}\Lambda_{\omega_{\varepsilon,\delta}}v,v\right)_{\omega_{\varepsilon,\delta},h_{\nu}}.
 \end{equation} 
By Proposition \ref{prop:BKNE}, if $v\in \mathrm{Dom}\bar{\partial}\cap\mathrm{Dom}\bar{\partial}^{\ast}$, then
\begin{equation}
\left(\mathrm{i}\Theta_{L,h_{\nu}}\Lambda_{\omega_{\varepsilon,\delta}}v,v\right)_{\omega_{\varepsilon,\delta},h_{\nu}}\leqslant
\|\bar{\partial}v\|_{\omega_{\varepsilon,\delta},h_{\nu}}^{2}+\|\bar{\partial}^{\ast}v\|_{\omega_{\varepsilon,\delta},h_{\nu}}^{2}.
\end{equation}
Comparing the above two inequality, we have
\begin{equation}
 \left|(f,v)_{\omega_{\varepsilon,\delta},h_{\nu}}\right|^{2}\leqslant\beta_{\nu}C\|v\|_{\omega_{\varepsilon,\delta},h_{\nu}}^{2}+\frac{C}{q}\left(\|\bar{\partial}v\|_{\omega_{\varepsilon,\delta},h_{\nu}}^{2}+\|\bar{\partial}^{\ast}v\|_{\omega_{\varepsilon,\delta},h_{\nu}}^{2}\right).
\end{equation}
Now, for any $v\in\mathrm{Dom}\bar{\partial}^{\ast}$, let us write
\[v=v_{1}+v_{2},\quad v_{1}\in\mathrm{Ker}\bar{\partial},\quad v_{2}\in\left(\mathrm{Ker}\bar{\partial}\right)^{\perp}\subset\mathrm{Ker}\bar{\partial}^{\ast}.\]
Then $\bar{\partial}v_{1}=0$ and $\bar{\partial}^{\ast}v_{2}=0$.
Since $\bar{\partial}f$=0, we get
\begin{equation}
\begin{split}
 \left|(v,f)_{\omega_{\varepsilon,\delta},h_{\nu}}\right|^{2}
 =&\left|(v_{1},f)_{\omega_{\varepsilon,\delta},h_{\nu}}\right|^{2} \\
 \leqslant&\beta_{\nu}C\|v_{1}\|_{\omega_{\varepsilon,\delta},h_{\nu}}^{2}+\frac{C}{q}\|\bar{\partial}^{\ast}v_{1}\|_{\omega_{\varepsilon,\delta},h_{\nu}}^{2} \\
 \leqslant&\beta_{\nu}C\|v\|_{\omega_{\varepsilon,\delta},h_{\nu}}^{2}+\frac{C}{q}\|\bar{\partial}^{\ast}v\|_{\omega_{\varepsilon,\delta},h_{\nu}}^{2} \\
 =&\frac{C}{q}\left(\left\|\sqrt{q\beta_{\nu}}\cdot v\right\|_{\omega_{\varepsilon,\delta},h_{\nu}}^{2}+\left\|\bar{\partial}^{\ast}v\right\|_{\omega_{\varepsilon,\delta},h_{\nu}}^{2} \right).
 \end{split}
\end{equation}
The Hahn-Banach theorem shows that the bounded linear functional
\begin{equation}
\left(\bar{\partial}^{\ast}v,\sqrt{q\beta_{\nu}}\cdot v\right)\mapsto\left(v,f\right).
\end{equation}
can be extended to a linear functional on
\[L_{n,q}^{2}\left(X_{c}\backslash Z_{\nu},L,\omega_{\varepsilon,\delta},h_{\nu}\right)\times L_{n,q}^{2}\left(X_{c}\backslash Z_{\nu},L,\omega_{\varepsilon,\delta},h_{\nu}\right)\]
whose norm is bounded by $\sqrt{\frac{C}{q}}$.
So we can find $u_{\varepsilon,c,\nu,\delta}$ and $w_{\varepsilon,c,\nu,\delta}$ such that
\begin{equation}
\left(\bar{\partial}^{\ast}v,u_{\varepsilon,c,\nu,\delta}\right)+\left(\sqrt{q\beta_{\nu}}\cdot v,w_{\varepsilon,c,\nu,\delta}\right)=\left(v,f\right),
\end{equation}
and
\begin{equation}
\left\|u_{\varepsilon,c,\nu,\delta}\right\|_{\omega_{\varepsilon,\delta},h_{\nu}}^{2}+\left\|w_{\varepsilon,c,\nu,\delta}\right\|_{\omega_{\varepsilon,\delta},h_{\nu}}^{2}\leqslant\frac{C}{q}.
\end{equation}
Therefore
\begin{equation}\label{eq:dbar4p}
\bar{\partial}u_{\varepsilon,c,\nu,\delta}+\sqrt{q\beta_{\nu}}\cdot w_{\varepsilon,c,\nu,\delta}=f,
\end{equation}
and 
\begin{equation}
\int_{X_{c}\backslash Z_{\nu}}\left|u_{\varepsilon,c,\nu,\delta}\right|^{2}_{\omega_{\varepsilon,\delta},h_{\nu}}\mathrm{d}V_{\omega_{\varepsilon,\delta}}+\int_{X_{c}\backslash Z_{\nu}}\left|w_{\varepsilon,c,\nu,\delta}\right|^{2}_{\omega_{\varepsilon,\delta},h_{\nu}}\mathrm{d}V_{\omega_{\varepsilon,\delta}}\leqslant\frac{C}{q}.
\end{equation}
Since the norm $\|\cdot\|_{\omega_{\varepsilon,\delta},h_{\nu}}$ is increasing as $\delta\to 0$, these uniform bounds imply that there are subsequences $u_{\varepsilon,c,\nu,\delta_{j}}$ and $w_{\varepsilon,c,\nu,\delta_{j}}$ with $\delta_{j}\to 0$, possessing weak-$L^{2}$ limits
\[u_{\varepsilon,c,\nu}=\lim_{j\to+\infty}u_{\varepsilon,c,\nu,\delta_{j}},\quad w_{\varepsilon,c,\nu}=\lim_{j\to+\infty}w_{\varepsilon,c,\nu,\delta_{j}},\]
and
\begin{equation}
\left\|u_{\varepsilon,c,\nu}\right\|_{\omega_{\varepsilon,\delta},h_{\nu}}^{2}+\left\|w_{\varepsilon,c,\nu}\right\|_{\omega_{\varepsilon,\delta},h_{\nu}}^{2}\leqslant\frac{C}{q}.
\end{equation}
Letting $\delta\to 0$, we obtain
\begin{equation}\label{eq:estimate3p}
\left\|u_{\varepsilon,c,\nu}\right\|_{\omega_{\varepsilon},h_{\nu}}^{2}+\left\|w_{\varepsilon,c,\nu}\right\|_{\omega_{\varepsilon},h_{\nu}}^{2}\leqslant\frac{C}{q},
\end{equation}
where $\omega_{\varepsilon}=\sigma+\varepsilon\omega$ and hence
$u_{\varepsilon,c,\nu}\in L_{n,q-1}^{2}\left(X_{c},L,\omega_{\varepsilon},h_{\nu}\right)$, $w_{\varepsilon,c,\nu}\in L_{n,q}^{2}\left(X_{c},L,\omega_{\varepsilon},h_{\nu}\right)$.
Then the equation \eqref{eq:dbar4p} implies that
\begin{equation}\label{eq:dbar3p}
\bar{\partial}u_{\varepsilon,c,\nu}+\sqrt{q\beta_{\nu}}\cdot w_{\varepsilon,c,\nu}=f
\end{equation}
on $X_{c}\backslash Z_{\nu}$.
By Lemma \ref{lem:extension}, the above equation can be extended to $X_{c}$.
By \eqref{eq:estimate3p}, we can choose weakly convergent subsequences and obtain weak limits
\[u_{\varepsilon,c}=\lim_{j\to+\infty}u_{\varepsilon,c,\nu_{j}},\quad w_{\varepsilon,c}=\lim_{j\to+\infty}w_{\varepsilon,c,\nu_{j}},\]
with estimates
\begin{equation}
\left\|u_{\varepsilon,c}\right\|_{\omega_{\varepsilon},h_{\nu}}^{2}+\left\|w_{\varepsilon,c}\right\|_{\omega_{\varepsilon},h_{\nu}}^{2}\leqslant\frac{C}{q},
\end{equation}
Letting $\nu\to+\infty$, we get
\begin{equation}
\left\|u_{\varepsilon,c}\right\|_{\omega_{\varepsilon},h}^{2}+\left\|w_{\varepsilon,c}\right\|_{\omega_{\varepsilon},h}^{2}\leqslant\frac{C}{q},
\end{equation}
Then by \eqref{eq:dbar3p},
\begin{equation}\label{eq:dbar2p}
\bar{\partial}u_{\varepsilon,c}=f
\end{equation}
since $\beta_{\nu}\to 0$ by construction.
We get therefore solutions $u_{\varepsilon,c}$ on $X_{c}$ with uniform $L^{2}$ bounds. Again, we can extract a weakly convergent subsequence to obtain
$u_{\varepsilon}$ such that
\[\bar{\partial}u_{\varepsilon}=f \quad\text{and}\quad\left\|u_{\varepsilon}\right\|_{\omega_{\varepsilon},h}^{2}\leqslant\frac{C}{q}.\]
Finally we choose a weakly convergent subsequence of $u_{\varepsilon}$ to get   a weak limit $u$ such that $\bar{\partial}u=f$ and 
\begin{equation}
\lim_{\varepsilon\to0}\int_{X}\left|u\right|^{2}_{\sigma+\varepsilon\omega,h}\mathrm{d}V_{\sigma+\varepsilon\omega}\leqslant\lim_{\varepsilon\to 0}\int_{X}\frac{1}{q\alpha}\left|f\right|^{2}_{\sigma+\varepsilon\omega,h}\mathrm{d}V_{\sigma+\varepsilon\omega}.
\end{equation}
\end{proof}

\begin{rmk}
In fact, we can only assume $X$ is complete K\"ahler instead of weakly pseudoconvex K\"ahler. 
Indeed, this fact can be proved by using the above arguments and the methods in \cite{Dem82}(p.475, Theorem 5.1).
We will not give the details here since Theorem \ref{thm:existence} is enough for our applications in next section, \end{rmk}


\section{Ohsawa-Nadel type vanishing theorem}

Let $Y$ be a paracompact reduced complex analytic space. 
A function $\eta$ on $Y$ is said to be smooth if for any point $y\in Y$, there exist a neighborhood $V$, a holomorphic embedding $\iota: V\hookrightarrow\mathbb{C}^{N}$ and a smooth function $\tilde{\eta}$ defined on a neighborhood of $\iota(V)$ such that $\eta=\iota^{\ast}\tilde{\eta}$.
By a hermitian metric on $Y$, we mean a hermitian metric $\sigma$ defined on the regular points of Y satisfying the following property: 
for any point $y\in Y$, there exist a neighborhood $V$, a holomorphic embedding $\iota: V\hookrightarrow\mathbb{C}^{N}$ and a smooth positive $(1,1)$-form $\tilde{\sigma}$ defined on a neighborhood of $\iota(V)$ for which $\sigma=\iota^{\ast}\tilde{\sigma}$ on the regular points of $V$. 
We say $\sigma$ is a K\"ahler metric if we can choose $\tilde{\sigma}$ to be $d$-closed. 
For any holomorphic map $f: X\to Y$ from a complex manifold $X$, $f^{\ast}\sigma$ can be extended uniquely to a smooth semipositive $(1,1)$-form on $X$. 

\begin{lem}[\cite{Oh84}]\label{lem:pullbackdec}
 Let $\pi: X\to Y$ be a holomorphic map between complex manifolds X and Y provided with hermitian metrics $\omega_{X}$ and $\omega_{Y}$ respectively. Then, for any form $g$ on $Y$,
\[\left|(\pi^{\ast}g)_{x}\right|_{\omega_{X}+\pi^{\ast}\omega_{Y}}\leqslant |g_{\pi(x)}|_{\omega_{Y}}\]
at any point $x\in X$.
\end{lem}

The purpose of this section is to study the cohomology groups.
Using the ideal of Ohsawa, we can prove the following result.

\begin{thm}
Let $X$ be a weakly pseudoconvex K\"ahler manifold.
let $f:X\to Y$ be a proper holomorphic map to a paracompact complex space $Y$ with a K\"ahler metric $\sigma$, and let $L$ be a line bundle over X equipped with a singular Hermitian metric. Assume that $\mathrm{i}\Theta_{L,h}\geqslant\delta\cdot f^{\ast}\sigma$ (in the sense of current) for some continuous positive function $\delta$ on $X$, then
\[H^{q}\left(Y,f_{\ast}\left(K_{X}\otimes L\otimes\mathcal{I}(h)\right)\right)=0\quad\text{for}\quad q\geqslant 1.\]
\end{thm}

\begin{proof}
Let $\mathcal{V}=\left\{V_{\alpha}\right\}_{\alpha\in I}$ be a locally finite Stein open covering of $Y$. We may assume that $V_{\alpha}\Subset Y$.
For the sake of simplicity, we denote
\[V_{\alpha_{0}\alpha_{1}\cdots\alpha_{q}}=V_{\alpha_{0}}\cap V_{\alpha_{1}}\cap\cdots\cap V_{\alpha_{q}}.\]
Let 
\[c=\left(c_{\alpha_{0}\alpha_{1}\cdots\alpha_{q}}\right)\in\prod_{(\alpha_{0},\alpha_{1},\cdots,\alpha_{q})\in I^{q+1}}
\Gamma\left(V_{\alpha_{0}\alpha_{1}\cdots\alpha_{q}}, f_{\ast}\left(K_{X}\otimes L\otimes\mathcal{I}(h)\right)\right)\] 
be a $q$-cocycle of $f_{\ast}\left(K_{X}\otimes L\otimes\mathcal{I}(h)\right)$ associated to $\mathcal{V}$.
We set
\[c^{\ast}_{\alpha_{0}\alpha_{1}\cdots\alpha_{q}}=f^{\ast}c_{\alpha_{0}\alpha_{1}\cdots\alpha_{q}},\quad U_{\alpha_{0}\alpha_{1}\cdots\alpha_{q}}=f^{-1}\left(V_{\alpha_{0}\alpha_{1}\cdots\alpha_{q}}\right).\]
Then $\mathcal{U}=\left\{f^{-1}\left(V_{\alpha}\right)\right\}_{\alpha\in I}$ is an open covering of $X$ and
\[\left(c^{\ast}_{\alpha_{0}\alpha_{1}\cdots\alpha_{q}}\right)\in\prod_{(\alpha_{0},\alpha_{1},\cdots,\alpha_{q})\in I^{q+1}}\Gamma\left(U_{\alpha_{0}\alpha_{1}\cdots\alpha_{q}}, K_{X}\otimes L\otimes\mathcal{I}(h)\right)\] 
is a cocycle of $K_{X}\otimes L\otimes\mathcal{I}(h)$ associated to $\mathcal{U}$.
Let $\eta_{\nu}\in\mathcal{C}^{\infty}_{0}(Y)$ be a partition of unity associated to $\mathcal{V}$. Then $\rho_{\nu}=f^{\ast}\eta_{\nu}\in\mathcal{C}^{\infty}_{0}(X)$ is a partition of unity associated to $\mathcal{U}$.
Let
\begin{equation}
b_{\alpha_{0}\alpha_{1}\cdots\alpha_{q-1}}=\sum_{\nu\in I}\rho_{\nu}c^{\ast}_{\nu\alpha_{0}\cdots\alpha_{q-1}}.
\end{equation}
be smooth $(n,0)$-forms with support in $U_{\alpha_{0}\cdots\alpha_{q-1}}$.
Since $\left(c^{\ast}_{\alpha_{0}\alpha_{1}\cdots\alpha_{q}}\right)$ is a cocycle, we have
\[c^{\ast}_{\alpha_{0}\alpha_{1}\cdots\alpha_{q}}=\sum_{j=0}^{q}(-1)^{j}c^{\ast}_{\nu\alpha_{0}\cdots\widehat{\alpha}_{j}\cdots\alpha_{q}} \quad\text{on}\quad U_{\nu\alpha_{0}\cdots\alpha_{q}}.\]
Then
\begin{equation}
\begin{split}
\sum_{j=0}^{q}(-1)^{j}b_{\alpha_{0}\cdots\widehat{\alpha}_{j}\cdots\alpha_{q}}
&=\sum_{j=0}^{q}(-1)^{j}\sum_{\nu\in I}\rho_{\nu}c^{\ast}_{\nu\alpha_{0}\cdots\alpha_{q-1}} \\
&=\sum_{\nu\in I}\rho_{\nu}\sum_{j=0}^{q}(-1)^{j}c^{\ast}_{\nu\alpha_{0}\cdots\alpha_{q-1}} \\
&=\sum_{\nu\in I}\rho_{\nu}c^{\ast}_{\alpha_{0}\alpha_{1}\cdots\alpha_{q}} \\
&=c^{\ast}_{\alpha_{0}\alpha_{1}\cdots\alpha_{q}}.
\end{split}
\end{equation}
So we get
\begin{equation}\label{eq:wsolvec}
c^{\ast}_{\alpha_{0}\alpha_{1}\cdots\alpha_{q}}=\sum_{j=0}^{q}(-1)^{j}b_{\alpha_{0}\cdots\widehat{\alpha}_{j}\cdots\alpha_{q}},
\end{equation}
and
\begin{equation}
\sum_{j=0}^{q}(-1)^{j}\bar{\partial}b_{\alpha_{0}\cdots\widehat{\alpha}_{j}\cdots\alpha_{q}}=0
\end{equation}
since $c^{\ast}_{\alpha_{0}\alpha_{1}\cdots\alpha_{q}}$ is $\bar{\partial}$-closed.
For $1\leqslant k\leqslant q-1$, we can obtain inductively smooth $(n,k)$-forms on $U_{\alpha_{0}\cdots\alpha_{q-k-1}}$
\begin{equation}\label{eq:defbalpha}
b_{\alpha_{0}\alpha_{1}\cdots\alpha_{q-k-1}}=\sum_{\nu\in I}\rho_{\nu}\bar{\partial}b_{\nu\alpha_{0}\cdots\alpha_{q-k-1}}
\end{equation}
such that
\begin{equation}\label{eq:dbarb}
\bar{\partial}b_{\alpha_{0}\alpha_{1}\cdots\alpha_{q-k}}=\sum_{j=0}^{q-k}(-1)^{j}b_{\alpha_{0}\cdots\widehat{\alpha}_{j}\cdots\alpha_{q-k}},
\end{equation}
and
\begin{equation}
\sum_{j=0}^{q-k}(-1)^{j}\bar{\partial}b_{\alpha_{0}\cdots\widehat{\alpha}_{j}\cdots\alpha_{k}}=0.
\end{equation}
When $k=q-1$,
\[\bar{\partial}b_{\alpha_{0}}-\bar{\partial}b_{\alpha_{0}}=0,\]
so we can define a smooth $(n,q)$-form $b$ on $X$ by
\[b=\bar{\partial}b_{\alpha_{0}}, \quad \forall\alpha_{0}\in I.\]
By \eqref{eq:defbalpha}, we have
\begin{equation}
b_{\alpha_{0}\cdots\alpha_{q-k-1}}=\sum_{\nu_{0},\cdots,\nu_{k}\in I}
\rho_{\nu_{k}}\bar{\partial}\rho_{\nu_{k-1}}\wedge\cdots\wedge\bar{\partial}\rho_{\nu_{0}}\wedge c^{\ast}_{\nu_{0}\nu_{1}\cdots\nu_{k}\alpha_{0}\cdots\alpha_{q-k-1}},
\end{equation}
and hence
\begin{equation}
b=\sum_{\nu_{0},\cdots,\nu_{q-1}\in I}
\bar{\partial}\rho_{\nu_{q-1}}\wedge\cdots\wedge\bar{\partial}\rho_{\nu_{0}}\wedge c^{\ast}_{\nu_{0}\nu_{1}\cdots\nu_{q-1}\alpha_{0}}.
\end{equation}

Let $\omega$ be a K\"ahler metric on $X$.
Since $c^{\ast}_{\alpha_{0}\alpha_{1}\cdots\alpha_{q}}$ are holomorphic $(n,0)$-forms, 
\[\left|c^{\ast}_{\alpha_{0}\alpha_{1}\cdots\alpha_{q}}\right|^{2}_{\varepsilon\omega+f^{\ast}\sigma,h}\mathrm{d}V_{\varepsilon\omega+f^{\ast}\sigma}\]
is independent of $\varepsilon>0$ and
\begin{equation}
\int_{K}\left|c^{\ast}_{\alpha_{0}\alpha_{1}\cdots\alpha_{q}}\right|^{2}_{\varepsilon\omega+f^{\ast}\sigma,h}\mathrm{d}V_{\varepsilon\omega+f^{\ast}\sigma}<+\infty
\end{equation}
for $K\Subset U_{\alpha_{0}\cdots\alpha_{q}}$ by definition of $\mathcal{I}(h)$.
Locally, $\eta_{\nu}$ (resp. $\sigma$) can be seen as the pull back of a smooth function (resp. K\"ahler form) on an open subset of $\mathbb{C}^{n}$.
By Lemma \ref{lem:pullbackdec}, $\left|\bar{\partial}\eta_{\nu}\right|_{\sigma}$ are locally bounded above. 
Then, again by Lemma \ref{lem:pullbackdec}, $\left|\bar{\partial}\rho_{\nu}\right|_{\varepsilon\omega+f^{\ast}\sigma}=\left|f^{\ast}\bar{\partial}\eta_{\nu}\right|_{\varepsilon\omega+f^{\ast}\sigma}$ are bounded by $\left|\bar{\partial}\eta_{\nu}\right|_{\sigma}$ which are independent of $\varepsilon>0$.
Therefore, for $1\leqslant k\leqslant q$ and $K\Subset U_{\alpha_{0}\cdots\alpha_{q-k}}$,
\begin{equation}\label{eq:localfinite}
\int_{K}\left|b_{\alpha_{0}\cdots\alpha_{q-k}}\right|^{2}_{\varepsilon\omega+f^{\ast}\sigma,h}\mathrm{d}V_{\varepsilon\omega+f^{\ast}\sigma}\leqslant C(K)<+\infty,
\end{equation}
and for $L\Subset X$
\begin{equation}
\int_{L}\left|b\right|^{2}_{\varepsilon\omega+f^{\ast}\sigma,h}\mathrm{d}V_{\varepsilon\omega+f^{\ast}\sigma}\leqslant C(L)<+\infty,
\end{equation}
where the constants $C(K)$ and $C(L)$ are independent of $\varepsilon>0$.

Since $X$ is weakly pseudoconvex,  there exists a smooth psh exhaustion function $\phi$ on $X$. Then we can also find a smooth convex increasing function $\chi$ such that
\begin{equation}
\int_{X}\frac{1}{\delta}\left|b\right|^{2}_{\varepsilon\omega+f^{\ast}\sigma,h}e^{-\chi\circ\phi}\mathrm{d}V_{\varepsilon\omega+f^{\ast}\sigma}\leqslant C<+\infty.
\end{equation}
where $C$ is independent of $\varepsilon$. Then
\begin{equation}
\lim_{\varepsilon\to 0}\int_{X}\frac{1}{\delta}\left|b\right|^{2}_{\varepsilon\omega+f^{\ast}\sigma,h}e^{-\chi\circ\phi}\mathrm{d}V_{\varepsilon\omega+f^{\ast}\sigma}<+\infty.
\end{equation}
By assumption 
\begin{equation}
\mathrm{i}\Theta_{L,he^{-\chi\circ\phi}}=\mathrm{i}\Theta_{L,h}+\mathrm{i}\partial\bar{\partial}(\chi\circ\phi)\geqslant\delta\cdot f^{\ast}\sigma,
\end{equation} 
so we can apply Theorem \ref{thm:existence} to find
\begin{equation}\label{eq:localfinite0}
a\in L^{2}_{n,q-1}\left(X,L,f^{\ast}\sigma,he^{-\chi\circ\phi}\right).
\end{equation}
satisfying \[\bar{\partial}a=b.\]
If we define
\begin{equation}
c^{\ast}_{\alpha_{0}}=b_{\alpha_{0}}-a,\quad\text{on}\quad U_{\alpha_{0}},\quad\forall\alpha_{0}\in I,
\end{equation}
then
\begin{equation}
\begin{cases}
\bar{\partial}c^{\ast}_{\alpha_{0}}=\bar{\partial}b_{\alpha_{0}}-\bar{\partial}a=0, \\
\bar{\partial}b_{\alpha_{0}\alpha_{1}}=b_{\alpha_{1}}-b_{\alpha_{0}}=c^{\ast}_{\alpha_{1}}-c^{\ast}_{\alpha_{0}}.
\end{cases}
\end{equation}
Since $V_{\alpha_{0}}$ is Stein, $U_{\alpha_{0}}=f^{-1}\left(V_{\alpha_{0}}\right)$
is holomorphically convex and it admits a smooth psh exhaustion function $\phi_{\alpha_{0}}$. By \eqref{eq:localfinite} and \eqref{eq:localfinite0}, we can find a smooth convex increasing function $\chi_{\alpha_{0}}$ such that
\[a,~ b_{\alpha_{0}}\in L^{2}_{n,q-1}\left(U_{\alpha_{0}},L,f^{\ast}\sigma,he^{-\chi_{\alpha_{0}}\circ\phi_{\alpha_{0}}}\right),\]
and hence
\begin{equation}
c^{\ast}_{\alpha_{0}}\in L^{2}_{n,q-1}\left(U_{\alpha_{0}},L,f^{\ast}\sigma,he^{-\chi_{\alpha_{0}}\circ\phi_{\alpha_{0}}}\right).
\end{equation}
Assume that $c^{\ast}_{\alpha_{0}\alpha_{1}\cdots\alpha_{q-k-1}}$ ($2\leqslant k\leqslant q-1$) are already obtained in such a way that
\begin{equation*}
\begin{cases}
\bar{\partial}c^{\ast}_{\alpha_{0}\alpha_{1}\cdots\alpha_{q-k-1}}=0, \\
\bar{\partial}b_{\alpha_{0}\alpha_{1}\cdots\alpha_{q-k}}=\sum_{j=0}^{q-k}(-1)^{j}c^{\ast}_{\alpha_{0}\cdots\widehat{\alpha}_{j}\cdots\alpha_{q-k}}, \\
c^{\ast}_{\alpha_{0}\cdots\alpha_{q-k-1}}\in L^{2}_{n,k}\left(U_{\alpha_{0}\cdots\alpha_{q-k-1}},L,f^{\ast}\sigma,he^{-\chi_{\alpha_{0}\cdots\alpha_{q-k-1}}\circ\phi_{\alpha_{0}\cdots\alpha_{q-k-1}}} \right).
\end{cases}
\end{equation*}
Since $\frac{1}{\delta}$ is bounded on $U_{\alpha_{0}\cdots\alpha_{q-k-1}}\Subset X$, again by Theorem \ref{thm:existence}, one can find
\[a_{\alpha_{0}\cdots\alpha_{q-k-1}}\in L^{2}_{n,k-1}\left(U_{\alpha_{0}\cdots\alpha_{q-k-1}},L,f^{\ast}\sigma,he^{-\chi_{\alpha_{0}\cdots\alpha_{q-k-1}}\circ\phi_{\alpha_{0}\cdots\alpha_{q-k-1}}}\right)\]
such that
\begin{equation}
\bar{\partial}a_{\alpha_{0}\cdots\alpha_{q-k-1}}=c^{\ast}_{\alpha_{0}\cdots\alpha_{q-k-1}}.
\end{equation}
Let
\begin{equation}
c^{\ast}_{\alpha_{0}\alpha_{1}\cdots\alpha_{q-k}}=b_{\alpha_{0}\alpha_{1}\cdots\alpha_{q-k}}-\sum_{j=0}^{q-k}(-1)^{j}a_{\alpha_{0}\cdots\widehat{\alpha}_{j}\cdots\alpha_{q-k}}.
\end{equation}
Then
\begin{equation}
\begin{split}
\bar{\partial}c^{\ast}_{\alpha_{0}\alpha_{1}\cdots\alpha_{q-k}}
&=\bar{\partial}b_{\alpha_{0}\alpha_{1}\cdots\alpha_{q-k}}-\sum_{j=0}^{q-k}(-1)^{j}\bar{\partial}a_{\alpha_{0}\cdots\widehat{\alpha}_{j}\cdots\alpha_{q-k}} \\
&=\bar{\partial}b_{\alpha_{0}\alpha_{1}\cdots\alpha_{q-k}}-\sum_{j=0}^{q-k}(-1)^{j}c^{\ast}_{\alpha_{0}\cdots\widehat{\alpha}_{j}\cdots\alpha_{q-k}} \\
&=0.
\end{split}
\end{equation}
It is easy to check
\begin{equation}
\begin{split}
\bar{\partial}b_{\alpha_{0}\alpha_{1}\cdots\alpha_{q-k+1}}
&=\sum_{j=0}^{q-k+1}(-1)^{j}b_{\alpha_{0}\cdots\widehat{\alpha}_{j}\cdots\alpha_{q-k+1}} \\
&=\sum_{j=0}^{q-k+1}(-1)^{j}c^{\ast}_{\alpha_{0}\cdots\widehat{\alpha}_{j}\cdots\alpha_{q-k+1}},
\end{split}
\end{equation}
where the first equality follows from \eqref{eq:dbarb}.
Since $V_{\alpha_{0}\cdots\alpha_{q-k}}$ is also Stein, 
there exists a smooth psh exhaustion function $\phi_{\alpha_{0}\cdots\alpha_{q-k}}$ on $U_{\alpha_{0}\cdots\alpha_{q-k}}=f^{-1}\left(V_{\alpha_{0}\cdots\alpha_{q-k}}\right)$.
By \eqref{eq:localfinite}, we can find a smooth convex increasing function $\chi_{\alpha_{0}\cdots\alpha_{q-k}}$ such that
$a_{\alpha_{0}\cdots\alpha_{q-k-1}}$ and $b_{\alpha_{0}\alpha_{1}\cdots\alpha_{q-k}}$ belong to 
\[L^{2}_{n,k-1}\left(U_{\alpha_{0}\cdots\alpha_{q-k}},L,f^{\ast}\sigma,he^{-\chi_{\alpha_{0}\cdots\alpha_{q-k}}\circ\phi_{\alpha_{0}\cdots\alpha_{q-k}}}\right).\]
Then
\begin{equation}
c^{\ast}_{\alpha_{0}\alpha_{1}\cdots\alpha_{q-k}}\in L^{2}_{n,k-1}\left(U_{\alpha_{0}\cdots\alpha_{q-k}},L,f^{\ast}\sigma,he^{-\chi_{\alpha_{0}\cdots\alpha_{q-k}}\circ\phi_{\alpha_{0}\cdots\alpha_{q-k}}}\right).
\end{equation}
Therefore
\begin{equation*}
\begin{cases}
\bar{\partial}c^{\ast}_{\alpha_{0}\alpha_{1}\cdots\alpha_{q-k}}=0, \\
\bar{\partial}b_{\alpha_{0}\alpha_{1}\cdots\alpha_{q-k+1}}=\sum_{j=0}^{q-k+1}(-1)^{j}c^{\ast}_{\alpha_{0}\cdots\widehat{\alpha}_{j}\cdots\alpha_{q-k+1}}, \\
c^{\ast}_{\alpha_{0}\cdots\alpha_{q-k}}\in L^{2}_{n,k-1}\left(U_{\alpha_{0}\cdots\alpha_{q-k}},L,f^{\ast}\sigma,he^{-\chi_{\alpha_{0}\cdots\alpha_{q-k}}\circ\phi_{\alpha_{0}\cdots\alpha_{q-k}}} \right).
\end{cases}
\end{equation*}
Finally, we can apply Theorem \ref{thm:existence} to find 
\[a_{\alpha_{0}\cdots\alpha_{q-2}}\in L^{2}_{n,0}\left(U_{\alpha_{0}\cdots\alpha_{q-2}},L,f^{\ast}\sigma,he^{-\chi_{\alpha_{0}\cdots\alpha_{q-2}}\circ\phi_{\alpha_{0}\cdots\alpha_{q-2}}}\right)\]
such that
\begin{equation}
\bar{\partial}a_{\alpha_{0}\cdots\alpha_{q-2}}=c^{\ast}_{\alpha_{0}\cdots\alpha_{q-2}}.
\end{equation}
Let
\begin{equation}
c^{\ast}_{\alpha_{0}\alpha_{1}\cdots\alpha_{q-1}}=b_{\alpha_{0}\alpha_{1}\cdots\alpha_{q-1}}-\sum_{j=0}^{q-1}(-1)^{j}a_{\alpha_{0}\cdots\widehat{\alpha}_{j}\cdots\alpha_{q-1}}.
\end{equation}
Then $\bar{\partial}c^{\ast}_{\alpha_{0}\alpha_{1}\cdots\alpha_{q-1}}=0$ and
\begin{equation}
c^{\ast}_{\alpha_{0}\alpha_{1}\cdots\alpha_{q-1}}\in L^{2}_{n,0}\left(U_{\alpha_{0}\cdots\alpha_{q-1}},L,f^{\ast}\sigma,he^{-\chi_{\alpha_{0}\cdots\alpha_{q-1}}\circ\phi_{\alpha_{0}\cdots\alpha_{q-1}}}\right).
\end{equation}
It follows that
\begin{equation}
\left(c^{\ast}_{\alpha_{0}\alpha_{1}\cdots\alpha_{q-1}}\right)\in\prod_{(\alpha_{0},\alpha_{1},\cdots,\alpha_{q-1})\in I^{q}}\Gamma\left(U_{\alpha_{0}\alpha_{1}\cdots\alpha_{q-1}}, K_{X}\otimes L\otimes\mathcal{I}(h)\right).
\end{equation}
By \eqref{eq:wsolvec}, we have
\begin{equation}
c^{\ast}_{\alpha_{0}\alpha_{1}\cdots\alpha_{q}}=\sum_{j=0}^{q}(-1)^{j}b_{\alpha_{0}\cdots\widehat{\alpha}_{j}\cdots\alpha_{q}}=\sum_{j=0}^{q}(-1)^{j}c^{\ast}_{\alpha_{0}\cdots\widehat{\alpha}_{j}\cdots\alpha_{q}}.
\end{equation}
If $c_{\alpha_{0}\alpha_{1}\cdots\alpha_{q-1}}$ are sections of $f_{\ast}\left(K_{X}\otimes L\otimes\mathcal{I}(h)\right)$ over $V_{\alpha_{0}\alpha_{1}\cdots\alpha_{q-1}}$ determined by
\[c^{\ast}_{\alpha_{0}\alpha_{1}\cdots\alpha_{q-1}}=f^{\ast}c^{\ast}_{\alpha_{0}\alpha_{1}\cdots\alpha_{q-1}},\]
then
\begin{equation}
c_{\alpha_{0}\alpha_{1}\cdots\alpha_{q}}=\sum_{j=0}^{q}(-1)^{j}c_{\alpha_{0}\cdots\widehat{\alpha}_{j}\cdots\alpha_{q}}.
\end{equation}
\end{proof}


\section{Remarks}

Let us recall the celebrated positivity theorem of Berndtsson.
Let $(X,\omega)$ be a K\"ahler manifold of dimension $n+m$ and $(Y,\sigma)$ a Hermitian manifold of dimension $m$. Let $p:X\to Y$ be a proper holomorphic map with surjective differential.  Let $(L, h)$ be a holomorphic Hermitian line bundle over $X$.
Then one can obtain a Hermitian vector bundle $(E,\|\cdot\|)$, where $E=p_{\ast}\left(K_{X/Y}\otimes L\right)$ and the metric $\|\cdot\|$ is obtained by integrating over the fibers $X_{t}=p^{-1}(t)$. Assume that $E_{t}=H^{0}\left(X_{t},K_{X_{t}}\otimes L|_{X_{t}}\right)\neq 0$.

\begin{thm}[\cite{Ber09}]\label{thm:positivity}
If the curvature $\mathrm{i}\Theta_{L,h}\geqslant p^{\ast}\sigma$, then $(E,\|\cdot\|)$ is positive in the sense of Nakano.
\end{thm}

The reader is referred to Section 4 in \cite{Ber09}, where the theorem was proved in the case $\mathrm{i}\Theta_{L,h}>0$.
However, in order to obtain the Nakano positivity of $E$, it is enough to assume that $\mathrm{i}\Theta_{L,h}\geqslant p^{\ast}\sigma$ instead of $\mathrm{i}\Theta_{L,h}>0$. In fact, this can be checked by using formula (4.8) in \cite{Ber09} p.550.

The following result is a special case of Ohsawa's vanishing theorem (c.f. Theorem \ref{thm:vanishing}). We will show that this statement can be deduced from Nakano vanishing theorem and Berndtsson's theorem.
\begin{thm}
Let $X$ be a K\"ahler manifold and $(Y,\sigma)$ a Hermitian manifold.
Let $p:X\to Y$ be a proper holomorphic map with surjective differential.
Let $(L, h)$ be a holomorphic Hermitian line bundle over $X$.
Assume that $Y$ is weakly pseudoconvex and the curvature $\mathrm{i}\Theta_{L,h}\geqslant p^{\ast}\sigma$, then
\[H^{q}\left(Y, p_{\ast}(K_{X}\otimes L) \right)=0 \quad\text{for}\quad q\geqslant 1.\]
\end{thm}

\begin{proof}
If $H^{0}\left(X_{t},K_{X_{t}}\otimes L|_{X_{t}}\right)=0$ ($t\in Y$), the statement is trivial. So we may assume $\mathrm{rank} E\geqslant 1$.
By Theorem \ref{thm:positivity}, $E$ is Nakano positivity.
Since $Y$ is weakly pseudoconvex, by Nakano vanishing theorem, we have
\[H^{q}\left(Y, K_{Y}\otimes E\right)=0\quad\text{for}\quad q\geqslant 1.\]
The projection formula implies
\[K_{Y}\otimes E=K_{Y}\otimes p_{\ast}\left(K_{X/Y}\otimes L\right)=p_{\ast}\left(K_{X}\otimes L\right).\]
So we can conclude
\[H^{q}\left(Y, p_{\ast}(K_{X}\otimes L) \right)=0 \quad\text{for}\quad q\geqslant 1.\]
\end{proof}




\begin{thebibliography}{99}

\bibitem{An-Ve65} A. Andreotti, E. Vesentini,
Carleman estimates for the Laplace-Beltrami equation in complex manifolds, 
Publ. Math. I.H.E.S. 25 (1965) 81-130.

\bibitem{Ber09} B. Berndtsson, 
Curvature of vector bundles associated to holomorphic fibrations, 
Ann. of Math. 169 (2009), 531–560. 

\bibitem{Cao14} J. Y. Cao, 
Numerical dimension and a Kawamata-Viehweg-Nadel type vanishing theorem on compact K\"{a}hler manifolds, 
Compositio Math. 150 (2014), 1869-1902.

\bibitem{Dem82} J.-P. Demailly, 
Estimations $L^{2}$ pour l'op\'{e}rateur $\bar{\partial}$ d'un fibr\'{e} vectoriel holomorphe semi-positif au-dessus d'une vari\'{e}t\'{e} k\"{a}hl\'{e}rienne compl\`{e}te, Ann. Sci. \'{E}cole Norm. Sup (4) 15 (1982), 457-511.

\bibitem{Dem94} J.-P. Demailly,
Regularization of closed positive currents of type $(1,1)$ by the flow of a Chern connection, 
Actes du Colloque en l'honneur de P. Dolbeault (Juin 1992), \'edit\'e par H. Skoda et J.-M. Tr\'epreau, Aspects of Mathematics, Vol. E 26, Vieweg, 1994, 105–126.

\bibitem{Dem12} J.-P. Demailly,
Analytic Methods in Algebraic Geometry, 
Surveys of Modern Mathematics 1, International Press, Somerville, MA; Higher Education Press, Beijing, 2012.

\bibitem{Dembook} J.-P. Demailly, 
Complex analytic and differential geometry, 
electronically accessible at 
\url{http://www-fourier.ujf-grenoble.fr/~demailly/books.html}.

\bibitem{DPS00} J.-P. Demailly, Th. Peternell, M. Schneider,
Pseudo-effective line bundles on compact K\"{a}hler manifolds, 
Internat. J. Math. 12 (2001), 689–741.

\bibitem{Fuj18} O. Fujino,
Koll\'ar–Nadel type vanishing theorem,
Southeast Asian Bulletin of Mathematics (2018) 42: 643-646.

\bibitem{Guan-Zhou15a} Q. Guan, X. Zhou, 
A proof of Demailly's strong openness conjecture, 
Ann. of Math, 182 (2015), 605-616.

\bibitem{Guan-Zhou15b} Q. Guan, X. Zhou, 
Effectiveness of Demailly's strong openness conjecture and related problems, 
Invent. Math,  202 (2015), no. 2, 635-676.

\bibitem{Hor65} L. H\"{o}rmander,
$L^{2}$ estimaths and existence theorems for the $\bar{\partial}$ operators, 
Acta Math. 113 (1965), 567-588.

\bibitem{Kol86} J. Koll\'ar,
Higher direct images of dualizing sheaves. I, 
Ann. of Math. (2) 123 (1) (1986) 11–42.

\bibitem{Mat16} S. Matsumura, 
A vanishing theorem of Koll\'ar-Ohsawa type,
 Math. Ann. 366 (3–4) (2016) 1451–1465.

\bibitem{Nadel} A. M. Nadel, 
Multiplier ideal sheaves and existence of K\"{a}hler-Einstein metrics of positive scalar curvature, 
Ann. of math., 132 (1990), 613-625.

\bibitem{Oh84} T. Ohsawa,
Vanishing Theorems on complete K\"{a}hler manifolds,
Publ. RIMS, Kyoto Univ. 20 (1984), 21-38.




\end{thebibliography}
\end{document}